\documentclass[9pt,shortpaper,twoside,web,web]{ieeecolor}

\usepackage{generic}
\usepackage{cite}
\usepackage{amsmath,amssymb,amsfonts,mathtools,mathrsfs}
\usepackage{graphicx}
\usepackage{textcomp}
\usepackage[ruled,vlined]{algorithm2e}
\usepackage{algpseudocode}
\usepackage{color}
\usepackage{balance}
\usepackage{glossaries}
\usepackage{paralist}
\usepackage{pifont}
\usepackage{bbding}
\usepackage{xcolor}

\newcommand{\R}{\mathbb{R}}

\newcommand{\N}{\mathbb{N}}

\newcommand{\mc}[1]{\mathcal{#1}}

\newcommand{\diag}{\mathrm{diag}}

\newcommand{\col}{\mathrm{col}}

\newcommand{\bs}[1]{\boldsymbol{#1}}
\newcommand{\bsone}{\boldsymbol{1}}

\newtheorem{theorem}{Theorem}

\newtheorem{definition}{Definition}
\newtheorem{proposition}{Proposition}
\newtheorem{lemma}{Lemma}

\newtheorem{remark}{Remark}

\newtheorem{standing}{Standing Assumption}

\newcommand{\blue}[1]{{#1}}

\makeglossaries
\newacronym{LICQ}{LICQ}{linear independent constraint qualification}
\newacronym{GNEP}{GNEP}{generalized Nash equilibrium problem}
\newacronym{l-SE}{$\ell$-SE}{local Stackelberg equilibrium}
\newacronym{ADAL}{ADAL}{accelerated distributed augmented Lagrangian}
\newacronym{SCA}{SCA}{sequential convex approximation}
\newacronym{PEV}{PEV}{Plug-in Electric Vehicle}

\newglossaryentry{v-GNE}
{
	name={v-GNE},
	description={generalized variational Nash equilbrium},
	first={\glsentrydesc{v-GNE} (\glsentrytext{v-GNE})},
	plural={v-GNE},
	descriptionplural={generalized variational Nash equilbria},
	firstplural={\glsentrydescplural{v-GNE} (\glsentryplural{v-GNE})}
}

\newglossaryentry{MPCC}
{
	name={MPCC},
	description={mathematical program with complementarity constraints},
	first={\glsentrydesc{MPCC} (\glsentrytext{MPCC})},
	plural={MPCCs},
	descriptionplural={mathematical programs with complementarity constraints},
	firstplural={\glsentrydescplural{MPCC} (\glsentryplural{MPCC})}
}

\newglossaryentry{MPEC}
{
	name={MPEC},
	description={mathematical program with equilibrium constraints},
	first={\glsentrydesc{MPEC} (\glsentrytext{MPEC})},
	plural={MPECs},
	descriptionplural={mathematical programs with equilibrium constraints},
	firstplural={\glsentrydescplural{MPEC} (\glsentryplural{MPEC})}
}

\begin{document}
\title{Local Stackelberg equilibrium seeking in generalized aggregative games}
\author{Filippo Fabiani, Mohammad Amin Tajeddini, Hamed Kebriaei, \IEEEmembership{Senior Member, IEEE},\\ and Sergio Grammatico, \IEEEmembership{Senior Member, IEEE}
\thanks{F.~Fabiani is with the Department of Engineering Science, University of Oxford, OX1 3PJ, United Kingdom ({\tt filippo.fabiani@eng.ox.ac.uk}). S.~Grammatico is with the Delft Center for Systems and Control, TU Delft, The Netherlands ({\tt s.grammatico@tudelft.nl}). M.~A.~Tajeddini and H.~Kebriaei are with the School of Electrical and Computer Engineering, College of Engineering, University of Tehran, Iran ({\tt\{a.tajeddini, kebriaei\}@ut.ac.it}). 
This work was partially supported by the ERC under research project COSMOS (ERC-StG 802348).}}

\maketitle

\begin{abstract}
We propose a two-layer, semi-decentralized algorithm to compute a local solution to the Stackelberg equilibrium problem in aggregative games with coupling constraints. Specifically, we focus on a single-leader, multiple-follower problem, and after equivalently recasting the Stackelberg game as a \gls{MPCC}, we iteratively convexify a regularized version of the \gls{MPCC} as inner problem, whose solution generates a sequence of feasible descent directions for the original \gls{MPCC}. Thus, by pursuing a descent direction at every outer iteration, we establish convergence to a local Stackelberg equilibrium. Finally, the proposed algorithm is tested on a numerical case study involving a hierarchical instance of the charging coordination of \glspl{PEV}.

\end{abstract}


\section{Introduction}\label{sec:introduction}
Stackelberg equilibrium problems are very popular within the system-and-control community, since they offer a multi-agent, decision-making framework that enables to model not only ``horizontal'' but also ``vertical'' interdependent relationships among heterogeneous agents, which are therefore clustered into \textit{leaders} and \textit{followers}. The application domains of Stackelberg equilibrium problems are, indeed, numerous, spanning from wireless networks, telecommunications \cite{liu2019sinr}, and network security \cite{TongwenChen2018}, to demand response and energy management \cite{motalleb2019networked,chen2017stackelberg,mendoza2019online}, economics \cite{hirose2019comparing}, and traffic control \cite{groot2017hierarchical}. 

In its most general setting, a Stackelberg equilibrium problem between a leader and a set of followers can be formulated as a \gls{MPEC} \cite[\S 1.2]{luo1996mathematical} or, in some specific cases, as an \gls{MPCC} \cite{scheel2000mathematical}. 
\blue{Both \glspl{MPEC} and \glspl{MPCC} are usually challenging to solve. Specifically, they} are inherently ill-posed, nonconvex optimization problems, since typically there are no feasible solutions strictly lying in the interior of the feasible set\blue{, which may even be disconnected},  implying that any constraint qualification is violated at every feasible point \cite{jongen1991nonlinear}.
\blue{It follows that, in this context,} the basic convergence assumptions \blue{characterizing} standard constrained optimization algorithms are not satisfied. Therefore, available solution methods are either tailored to the specific problem considered, or designed \textit{ad hoc} for a sub-class of \glspl{MPEC}/\glspl{MPCC}.

Algorithmic solution techniques for \blue{the class of} games involving dominant and nondominant strategies, i.e. leaders and followers, trace back to the 70s. For example, open-loop and feedback control policies for differential, \blue{hence continuous-time}, unconstrained games were designed in \cite{simaan1973stackelberg,kydland1977equilibrium}, while in \cite{kydland1975noncooperative} a comparison between finite/infinite horizon control strategies involving discrete-time dynamics was proposed. More recently, a single-leader, multi-follower differential game, modeling a pricing scheme for the Internet by basing on the bandwidth usage of the users, i.e., with congestion constraints, was solved in \cite{bacsar2002stackelberg}, and an iterative procedure to compute a Stackelberg-Nash-saddle point for an unconstrained, single-leader, multi-follower game with discrete-time dynamics was proposed in \cite{kebriaei2017discrete}.
\blue{By relying on the uniqueness of the followers' equilibrium for each leader's strategy, standard fixed-point algorithms are also proposed in \cite{tushar2012economics,zou2017decentralized}.}
\blue{A first attempt to solve an \gls{MPEC} modelling} a more elaborated multi-leader, multi-follower game, was investigated in \cite{kulkarni2015existence}\blue{. Specifically, the authors established} the equivalence to a single-leader, multi-follower game whenever the cost functions of the leaders admit a potential function and, in addition, the set of leaders has an identical conjecture or estimate on the follower equilibrium. Similar arguments are also exploited in \cite{leyffer2010solving} to address the same multi-leader, multi-follower equilibrium problem. In this latter case, for each leader, the authors proposed a single-leader, multi-follower game \blue{modelled as an \gls{MPEC}}. On the other hand, all these sub-games, which are parametric in the decisions of the followers, are coupled together through a game \blue{against} the leaders themselves. However, in both papers the solution to the single-leader, multi-follower game remains to be dealt with, mainly due to the presence of nonconvexities and equilibrium/complementarity constraints which characterize \gls{MPEC}/\gls{MPCC}. 
\blue{Early algorithmic works on \glspl{MPCC}} to solve single-leader, multi-follower Stackelberg games, such as Gauss-Seidel or Jacobi \cite{hobbs2001linear,ehrenmann2009comparison}, are computationally expensive\blue{, especially for large number of followers. 
A}dditionally, they introduce several privacy issues, since they are designed by relying on diagonalization techniques.
\blue{In \cite{su2004sequential}, after relaxing the complementarity conditions, a solution to an \gls{MPCC} is computed through nonlinear complementarity problems, towards driving the relaxation parameter to zero.}

\blue{Our work aims at filling the} apparent lack \blue{in the aforementioned literature} of scalable and privacy preserving solution \blue{algorithms} for equilibrium problems with nonconvex data and complementarity conditions\blue{, i.e., \glspl{MPEC}/\glspl{MPCC}}. \blue{Specifically,} we leverage on the \gls{SCA} to design a two-layer, semi-decentralized algorithm suitable to iteratively compute a local solution to the Stackelberg equilibrium problem involving a single leader and multiple followers in aggregative form \blue{with coupling constraints}. The main contributions of the paper are summarized as follows:
\begin{itemize}
\item We reformulate the Stackelberg game as an \gls{MPCC} by embedding it into the leader nonconvex optimization problem the equivalent KKT conditions to compute a \gls{v-GNE} \cite{facchinei2007finite} for the followers' game (\S II);

\item We exploit a key result provided in \cite{scholtes2001convergence} to locally relax the complementarity constraints, obtaining the \gls{MPCC}-LICQ \cite[Def.~3.1]{fletcher2006local}, i.e., the \gls{LICQ} of all the points inside a certain neighborhood of the originally formulated \gls{MPCC} (\S III);

\item Along the same lines of \cite{scutari2014decomposition,scutari2017parallel}, we propose to convexify the relaxed \gls{MPCC} at every iteration of the outer loop, whose optimal solution, computed within the inner loop, points a descent direction for the cost function of the original \gls{MPCC}. By pursuing such a descent direction, the sequence of feasible points generated by the outer loop directly leads to a local solution of the Stackelberg equilibrium problem (\S III);

\item We analyze the performance of the proposed algorithm applied to a numerical instance of the charging coordination problem for a fleet of \glspl{PEV}, also investigating the behavior of the leader and the followers as the regularization parameter varies (\S IV).
\end{itemize}
To the best of our knowledge, the proposed two-layer algorithm represents the first attempt to compute a \blue{local} solution to the Stackelberg equilibrium problem involving nonconvex data and equilibrium constraints by directly exploiting (and preserving) the hierarchical, multi-agent structure of the original aggregative game.



\subsection*{Notation}
$\N$, $\R$ and $\R_{\geq 0}$ denote the set of natural, real and nonnegative real numbers. $\bsone$ represents a vector with all elements equal to $1$. For vectors $v_1,\dots,v_N\in\mathbb{R}^n$ and $\mc I=\{1,\dots,N \}$, we denote $\bs{v}:= (v_1 ^\top,\dots ,v_N^\top )^\top = \mathrm{col}(\{v_i\}_{i\in\mc I})$ and $ \bs{v}_{-i} \coloneqq \col(\{ v_j \}_{j\in\mc I\setminus \{i\}})$. We also use $\bs{v} = (v_i,\bs{v}_{-i})$. \blue{$v \perp w$ means that $v$ and $w$ are orthogonal vectors.} Given a matrix $A \in \R^{m \times n}$, $A^\top$ denotes its transpose. $A \otimes B$ represents the Kronecker product between the matrices $A$ and $B$. For a function $f : \R^n \times \R^n \to \R$, $f(v;\bar{v})$ denotes the approximation of $f$ at some $\bar{v}$. \blue{For a set-valued mapping $\mc{F} : \R^n \rightrightarrows \R^m$, $\textrm{gph}(\mc{F}) \coloneqq \{(y,x) \in \R^{n}\times \R^m \mid x \in \mc{F}(y)\}$ denotes its graph.}

\section{Mathematical setup}
\subsection{Stackelberg game}
We consider a hierarchical noncooperative game with one leader, controlling its decision variable $y_0 \in \R^{n_0}$, and  $N$ followers, indexed by the set $\mc{I} \coloneqq \{1, \ldots, N\}$, where each follower $i \in \mc{I}$ controls its own variable $x_i \in \mc{X}_i \coloneqq \left\{x_i \in \R^{n_i} \mid F_i x_i \leq g_i \right\}$, $F_i \in \R^{p_i \times n_i}$, $g_i \in \R^{p_i}$, and aims at solving the following optimization problem:
\begin{equation}\label{eq:single_prob}
\forall i \in \mc{I} : \left\{
\begin{aligned}
&\underset{x_i \in \mc{X}_i}{\textrm{min}} & & J_i (y_0, x_i,  \bs{x}_{-i})\\
&\hspace{.1cm}\textrm{ s.t. } & & A_i x_i + \textstyle\sum_{j \in \mc{I} \setminus \{i\}} A_j x_j \leq b,
\end{aligned}	
\right.
\end{equation}
for some cost function $J_i:\R^{n_0} \times \R^{n} \to \R$. Let $\bs{x} \coloneqq \col(\{x_i\}_{i \in \mc{I}}) \in \R^n$, $n = \sum_{i \in \mc{I}} n_i$, be the collective vector of strategies of the followers, while $\bs{x}_{-i} \in \R^{n - n_i}$ stacks all the local decision variables except the $i$-th one. We postulate the following standard assumptions on the followers' data in \eqref{eq:single_prob}.
\smallskip
\begin{standing}\label{ass:C_1}
	For each $i \in \mc{I}$, the function $J_i(y_0, \cdot)$ is convex and continuously differentiable, for fixed $y_0$.
	\hfill$\square$
\end{standing}
\smallskip
\begin{standing}\label{ass:full_row_rank}
	For each $i \in \mc{I}$, $\textrm{rank}(F_i) = p_i$.
	\hfill$\square$
\end{standing}
\smallskip

In \eqref{eq:single_prob}, each matrix $A_i \in \R^{m \times n_i}$ stacks $m$ linear coupling constraints, while $b \in \R^m$ is the vector of shared resources among the followers.
Let $A \coloneqq \left[A_1 \, \ldots \, A_N\right] \in \R^{m \times n}$. Then, we preliminary define the sets $\mc{X} \coloneqq \prod_{i \in \mc{I}} \mc{X}_i$ and $\Theta \coloneqq \{\bs{x} \in \mc{X} \mid A \bs{x} \leq b\}$.

For a fixed strategy of the leader, $y_0$, the followers aim to solve a \gls{GNEP}. Specifically, by focusing on \glspl{v-GNE}, such problem is equivalent to solve VI$(\Theta,H(y_0, \cdot))$ \cite{facchinei2007finite}, where, in view of Standing Assumption~\ref{ass:C_1}, $H : \R^{n_0} \times \R^n \rightrightarrows \R^n$ is a continuously differentiable set-valued mapping defined as $H\left(y_0, \bs{x}\right) \coloneqq \col(\{\nabla_{x_{i}} J_{i}\left(y_0, \bs{x}\right)\}_{i \in \mathcal{I}})$. \blue{This fact, along with the properties of $\Theta$, guarantee the nonemptiness of the set of \gls{v-GNE} that,} for any $y_0 \in \mc{Y}_0$, corresponds to the set
\begin{equation}\label{eq:set_vGNE}
	\mc{S}(y_0)\coloneqq \{\bs{x} \in \Theta \mid (\bs{z} - \bs{x})^\top H(y_0, \bs{x}) \geq 0, \; \forall \bs{z} \in \Theta\}.
\end{equation}
On the other hand, the optimization problem of the leader reads as:

\begin{equation}\label{eq:single_prob_leader}
	\left\{
	\begin{aligned}
		&\underset{y_0, \bs{x}}{\textrm{min}} & & J_0 (y_0, \bs{x})\\
		&\textrm{ s.t. } & & \blue{(y_0, \bs{x}) \in \textrm{gph}(\mc{S}) \cap (\mc{Y}_0 \times \R^n)},\\
	\end{aligned}	
	\right.
\end{equation}
for some cost function $J_0 : \R^{n_0} \times \R^n \to \R$ and local constraint set $\mc{Y}_0$ characterized by the following standard conditions.
\smallskip
\begin{standing}\label{ass:Y_0}
	The set $\mc{Y}_0$ is nonempty, closed and convex.
	\hfill$\square$
\end{standing}
\smallskip
\begin{standing}\label{ass:continuity}
		The function $J_0$ is coercive, its gradient $\nabla J_0$ is Lipschitz continuous on $\Phi \coloneqq \mc{Y}_0 \times \mc{X}$ with constant $\kappa_0$.
		\hfill$\square$
\end{standing}
\smallskip

\blue{
We note that \eqref{eq:single_prob_leader} defines an \gls{MPEC} where $\bs{x}$ is not strictly within the leader’s control, but it corresponds to an optimistic conjecture \cite{kulkarni2015existence}. In view of \cite[Th.~1.4.1]{luo1996mathematical}, the MPEC in \eqref{eq:single_prob_leader} admits an optimal solution, since the coerciveness of $J_0$ implies compactness of its level sets, and the feasible set, $\textrm{gph}(\mc{S}) \cap (\mc{Y}_0 \times \R^n)$, is closed under the postulated assumptions. Therefore, this ensures existence of a solution to the hierarchical game, according to the following notion of local generalized Stackelberg equilibrium, inspired by \cite{hu2007using,kulkarni2015existence}.
}
\smallskip
\begin{definition}\label{def:S-EQ}
	A pair $(y_0^{*}, \bs{x}^{*}) \in \blue{ \textrm{gph}(\mc{S}) \cap (\mc{Y}_0 \times \R^n)}$\blue{, with $\mathcal{S}$ as in \eqref{eq:set_vGNE}}, is a \gls{l-SE} of the hierarchical game in \eqref{eq:single_prob}--\eqref{eq:single_prob_leader} if there exist open neighborhoods $\mathcal{O}_{y_0^{*}}$ and $\mathcal{O}_{\bs{x}^{*}}$ of $y_0^{*}$ and $\bs{x}^{*}$, respectively, such that
	$$
		J_0(y_0^{*}, \bs{x}^{*}) \leq \underset{\blue{(y_0, \bs{x}) \in \textrm{gph}(\mc{S}) \cap \mathcal{O}}}{\mathrm{inf}} J_0(y_0, \bs{x}),
	$$
	\blue{where $\mathcal{O} \coloneqq (\mc{Y}_0 \cap \mathcal{O}_{y_0^{*}}) \times \mathcal{O}_{\bs{x}^{*}}$.}
	\hfill$\square$
\end{definition}
\smallskip
Informally speaking, at an \gls{l-SE}, the leader and the followers locally fulfill the set of mutually coupling constraints and none of them can gain by unilaterally deviating from their current strategy.
Note that we refer to an SE if Definition~\ref{def:S-EQ} holds true \blue{with $\mathcal{O} = \mc{Y}_0 \times \R^{n}$, i.e.,  $\mc{O}_{y_0^{*}} = \R^{n_0}$ and $\mc{O}_{\bs{x}^{*}} = \R^{n}$, thus coinciding with \cite[Def.~1.1]{kulkarni2015existence}}.

\subsection{Aggregative game formulation}
For computational purposes, we consider the cost function of the followers and leader to be in aggregative form, i.e.,
\begin{equation}\label{eq:cost_functions}
	\begin{aligned}
		&J_{i} \coloneqq \tfrac{1}{2} x_i^\top Q_i x_i + \left(\tfrac{1}{N} \textstyle\sum_{j \in \mc{I}} C_{i,j} x_j + C_{i,0} y_0\right)^{\top} \! x_{i}, \; \forall i \in \mathcal{I},\\ 
		&J_{0} \coloneqq f_{0}\left(y_{0}\right)+\left( \textstyle\sum_{i \in \mc{I}} f_{0,i}(x_i) \right)^{\top} \! y_{0},
	\end{aligned}
\end{equation}
where $Q_i$ $\blue{\succcurlyeq }$ $0$, $C_{i,j} \in \R^{n_i \times n_j}$, and $C_{i,0} \in \R^{n_i \times n_0}$. 
In view of Standing Assumption~\ref{ass:C_1}, given any feasible $y_0 \in \mc{Y}_0$, it follows from \cite[Th.~3.1]{facchinei2007gen} that a set of strategies is a \gls{v-GNE} of the followers game in \eqref{eq:single_prob} if and only if the following coupled KKT conditions hold true:
$$
\left\{\begin{array}{l}{\nabla_{x_{i}} J_{i}\left( y_{0},x_{i}, \bs{x}_{-i}\right)+A_{i}^{\top} \lambda +F_{i}^{\top} \lambda_i = 0}, \; \forall i \in \mathcal{I}, \\ 
{0 \leq \lambda \perp-(A \bs{x}-b) \geq 0},\\
{0 \leq \blue{\lambda_i} \perp-(F_i x_i-g_i) \geq 0}, \; \forall i \in \mathcal{I},
\end{array} \right.
$$
which, in our aggregative setup, can be compactly rewritten as
\begin{equation}\label{eq:KKT_followers}
	\left\{\begin{array}{l} 
	Q \bs{x} + C y_0 +A^{\top} \lambda + F^\top \bs{\lambda} = 0, \\ 
	{0 \leq \lambda \perp-(A \bs{x}-b) \geq 0},\\
	{0 \leq \blue{\lambda_i} \perp-(F_i x_i-g_i) \geq 0}, \; \forall i \in \mc{I},
	\end{array} \right.
\end{equation}
where $F \coloneqq \diag(\{F_i\}_{i \in \mc{I}})$, $\lambda \in \mathbb{R}_{ \geq 0}^{m}$ is the dual variable associated with $A \bs{x} \leq b$, $\lambda_i \in \mathbb{R}_{ \geq 0}^{p_i}$ is the (local) dual variable associated with the local constraints defining $\mc{X}_i$, $\bs{\lambda} \coloneqq \col(\left\{\lambda_i\right\}_{i \in \mc{I}})$, and
$$
Q \coloneqq \left[
\begin{array}{ccc}
Q_1 + \tfrac{1}{N}C_{1,1} & \cdots & \tfrac{1}{N}C_{1,N}\\
\vdots & \ddots & \vdots\\
\tfrac{1}{N}C_{N,1}  & \cdots & Q_N + \tfrac{1}{N} C_{N,N}
\end{array}
\right], \,
C \coloneqq \left[
\begin{array}{c}
C_{10}\\
\vdots\\
C_{N0}
\end{array}
\right].
$$

Finally, by substituting back the KKT conditions in \eqref{eq:KKT_followers} into the optimization problem of the leader in \eqref{eq:single_prob_leader}, the problem of finding an SE of the hierarchical game in \eqref{eq:single_prob}--\eqref{eq:single_prob_leader} can be equivalently written as

\begin{equation}\label{KKT_into_leader}
\left\{
\begin{aligned}
&\underset{y_0, \bs{x}, \lambda, \bs{\lambda}}{\textrm{min}} & & J_0(y_0, \bs{x})\\
&\hspace{.32cm}\textrm{ s.t. } & & Q \bs{x} + C y_0 +A^{\top} \lambda + F^\top \bs{\lambda} = 0,\\
&&& 0 \leq \lambda_i \perp-(F_i x_i-g_i) \geq 0, \; \forall i \in \mc{I},\\
&&& 0 \leq \lambda \perp-(A \bs{x}-b) \geq 0, \; y_0 \in \mc{Y}_0.
\end{aligned}	
\right.
\end{equation}

\subsection{Complementarity constraints relaxation}
We note that the leader nonconvex optimization problem in \eqref{KKT_into_leader} is an \gls{MPCC} and, in general, it does not satisfy any standard constraint qualification. 
Therefore, we propose to study a regularized version by introducing slack variables $\mu \in \mathbb{R}_{ \geq 0}^{m}$ and $\mu_i \in \mathbb{R}_{ \geq 0}^{p_i}$, $i \in \mc{I}$, together with parameters  $\theta, \, \theta_i > 0$, $i \in \mc{I}$, which enable us to replace the complementarity constraints in \eqref{KKT_into_leader} with the nonlinear constraints $\lambda^\top \mu \leq \theta$ and $\lambda_i^\top \mu_i \leq \theta_i$, for all $i \in \mc{I}$ \cite{scholtes2001convergence}.
Thus, after defining $\nu \coloneqq \col(\lambda, \mu) \in \R^{2m}$, $\nu_i \coloneqq \col(\lambda_i, \mu_i) \in \R^{2p_i}$, $\bs{y} \coloneqq \col(\bs{x}, \{\nu_i\}_{i \in \mc{I}})$, the regularized version of \eqref{KKT_into_leader} reads as:
\begin{equation}\label{eq:KKT_into_leader_relaxed}
R(\theta) : \left\{
\begin{aligned}
&\underset{y_0, \bs{y}, \nu}{\textrm{min}} & & J_0(y_0, \bs{x})\\
&\hspace{.15cm}\textrm{ s.t. } & & A_\textrm{f} \, \bs{y} + A_\ell \, y_0 + A_\textrm{c} \, \nu = d,\\
&&&   \lambda_i^\top \mu_i \leq \theta_i, \lambda_i, \mu_i \geq 0, \; \forall i \in \mc{I},\\
&&& \lambda^\top \mu \leq \theta, \lambda, \mu \geq 0, \; y_0 \in \mc{Y}_0,
\end{aligned}	
\right.
\end{equation}
where $d \coloneqq \col(0, b, g)$, $g \coloneqq \col(\{g_i\}_{i \in \mc{I}})$, $A_\ell \coloneqq \col(C, 0, 0)$, and
\begin{align*}
A_{\textrm{f}} \coloneqq \left[
\begin{array}{cc}
Q & \{[F^\top_i \;\; 0 ]\}_{i \in \mc{I}}\\
A & 0\\
F & \left[0 \;\; I\right] \otimes \bsone
\end{array}
\right], \; A_{\textrm{c}} \coloneqq \left[
\begin{array}{cc}
A^\top & 0\\
0 & I\\
0 & 0
\end{array}
\right].
\end{align*}

For any given $\theta, \, \theta_i > 0$, $i \in \mc{I}$, let us now introduce the sets 
\begin{equation}\label{eq:nonconvex_sets}
	\begin{aligned}
		\mc{C}(\theta) &\coloneqq \{ \nu \in \R^{2m}_{\geq 0} \mid \tfrac{1}{2} \nu^\top P \nu \leq \theta  \},\\ 
		\mc{C}_i(\theta_i) &\coloneqq \{ \nu_i \in \R^{2p_i}_{\geq 0} \mid \tfrac{1}{2} \nu_i^\top P_i \nu_i \leq \theta_i \}, \; \forall i \in \mc{I}.
	\end{aligned}
\end{equation}
Here, each $P$ and $P_i$, $i \in \mc{I}$, is a symmetric matrix with identities of suitable dimension on the anti-diagonal.
Furthermore, we define $\Omega(\theta) \coloneqq \mc{Y}_0 \times \mc{Y} \times \mc{C}(\theta)$, where for brevity we omit the dependency from $\theta_i$, explicated in $\mc{Y} \coloneqq \mc{X} \times \prod_{i \in \mc{I}} \mc{C}_i(\theta_i)$. Finally, by introducing $\bs{\omega} \coloneqq \col(y_0, \bs{y}, \nu)$ and $A_\omega \coloneqq [A_{\ell} \; A_\textrm{f} \;  A_\textrm{c}]$, the closed, nonconvex feasible set of $R(\theta)$ in \eqref{eq:KKT_into_leader_relaxed} reads as
\begin{equation}\label{eq:feasible_set}
	\mc{R}(\theta) \coloneqq \{\bs{\omega} \in \Omega(\theta) \mid A_{\omega} \, \bs{\omega} - d = 0\}.
\end{equation}

\blue{
	We recall now the notion of \gls{MPCC}-\gls{LICQ} for the \gls{MPCC} in \eqref{KKT_into_leader}, which is characterized by the result stated immediately below.
	\smallskip
	\begin{definition}
		The \gls{MPCC} in \eqref{KKT_into_leader} satisfies the \gls{MPCC}-\gls{LICQ} at $\tilde{\bs{\omega}} \in \mc{R}(0)$ if $R(0)$ in \eqref{eq:KKT_into_leader_relaxed} satisfies the \gls{LICQ} at $\tilde{\bs{\omega}}$.
		\hfill$\square$
	\end{definition}
}
\smallskip
\begin{lemma}(\hspace{-.01cm}\cite[Lemma~2.1]{scholtes2001convergence})\label{lemma:scholtes}
	Let $\tilde{\bs{\omega}} \in \mc{R}(0)$. If $\tilde{\bs{\omega}}$ satisfies the \gls{MPCC}-\gls{LICQ} for the \gls{MPCC} in \eqref{KKT_into_leader}, then there exists an open neighborhood $\mc{O}$ of $\tilde{\bs{\omega}}$ and scalars $\tilde{\theta}$, $\tilde{\theta}_i > 0$, for all $i \in \mc{I}$, such that, for every $\theta \in (0, \tilde{\theta})$ and $\theta_i \in (0, \tilde{\theta}_i)$, for all $i \in \mc{I}$, the \gls{LICQ} holds true at every point $\bs{\omega} \in \mc{O}$ of $R(\theta)$.
	\hfill$\square$
\end{lemma}
\smallskip

\blue{Then, let us introduce the following fundamental assumption.}
\smallskip
\begin{standing}\label{standing:existence_of_point}
	There exists some $\tilde{\bs{\omega}} \in \mc{R}(0)$ that satisfies the \gls{MPCC}-\gls{LICQ} for the \gls{MPCC} in \eqref{KKT_into_leader}. The regularization parameters are chosen so that $\theta \in (0, \tilde{\theta})$ and $\theta_i \in (0, \tilde{\theta}_i)$, for all $i \in \mc{I}$.
	\hfill$\square$
\end{standing}
\smallskip

In view of Standing Assumption~\ref{standing:existence_of_point}, there exists a neighborhood such that $R(\theta)$ locally satisfies the \gls{LICQ}. As \blue{shown} in \S \ref{subsec:tradeoff}, the coefficients $\theta$, $\theta_i$, $i \in \mc{I}$, play a trade-off role between the distance from a \gls{v-GNE} for the followers and a lower cost for the leader. To conclude \blue{the section}, we stress that an optimal solution to \eqref{eq:KKT_into_leader_relaxed}\blue{, whose existence follows by its local  \gls{LICQ} and the coerciveness of $J_0$,} generates a pair $(y^\ast_0, \bs{x}^\ast)$ that corresponds to an \gls{l-SE} of the original hierarchical game in \eqref{eq:single_prob}--\eqref{eq:single_prob_leader}.

\section{Local Stackelberg equilibrium seeking\\ via sequential convex approximation}
\subsection{A two-layer algorithm}
In the spirit of \cite{scutari2014decomposition,scutari2017parallel}, we then investigate how to solve \eqref{eq:KKT_into_leader_relaxed} in a decentralized fashion by means of a two-layer algorithm, while preserving the hierarchical structure of the game \eqref{eq:single_prob}--\eqref{eq:single_prob_leader}. 
First, we linearize the nonlinear terms appearing in the cost function around some $\bar{\bs{\omega}} \in \mc{R}(\theta)$. 
Specifically, with $\bs{\bs{\varphi}} \coloneqq (y_0, \bs{x})$, $J_0$ is linearized by following a first order Taylor expansion as $ J_0 (\bs{\varphi}) \simeq J_0 (\bar{\bs{\varphi}}) + \nabla^\top J_0 (\bar{\bs{\varphi}}) \left(\bs{\varphi} - \bar{\bs{\varphi}}\right)$ where, for our aggregative game, we have:
$$ 
	\begin{aligned}
		\nabla J_0(\bs{\bs{\varphi}}) &\!=\! \col(\nabla_{y_0} f_0(y_0) \!+\! \textstyle \sum_{j \in \mc{I}} f_{0,j}(x_j), \! \{\nabla_{x_{j}} f_{0,j}(x_j)\blue{^\top \! y_0}\}_{j \in \mc{I}})\\ 
		&\!\eqqcolon\! \col(c_{\ell}(y_0, \bs{x}), c_{\textrm{f}}(\blue{y_0, } \bs{x})).
	\end{aligned}
$$ 

According to \cite[\S III.A]{scutari2017parallel}, for the nonlinear constraints \blue{defining} the sets in \eqref{eq:nonconvex_sets}, we compute an upper approximation by observing that, e.g., $\blue{\tfrac{1}{2} \nu^\top P \nu} = \lambda^\top \mu = \tfrac{1}{2} (\lambda + \mu)^\top(\lambda + \mu) - \tfrac{1}{2} (\lambda^\top\lambda + \mu^\top\mu)$. Thus, after linearizing the concave term around some $\bar{\nu} \in \mc{C}(\theta)$, we define
$$
	\tilde{\mc{C}}(\theta; \bar{\bs{\omega}}) \coloneqq \{\nu \in \R^{2m}_{\geq 0} \mid \tfrac{1}{2} (\bsone^\top \nu)^\top(\bsone^\top \nu) - \bar{\nu}^\top \nu + \tfrac{1}{2} \bar{\nu}^\top\bar{\nu}  \leq \theta \}.
$$
The same procedure can be applied to each $\mc{C}_i(\theta_i)$ to obtain $\tilde{\mc{C}}_i(\theta_i; \bar{\bs{\omega}})$.
Accordingly, $\Omega(\theta)$ is approximated by $\tilde{\Omega}(\theta; \bar{\bs{\omega}}) \coloneqq \mc{Y}_0 \times \tilde{\mc{Y}}(\bar{\bs{\omega}}) \times  \tilde{\mc{C}}(\theta; \bar{\bs{\omega}})$, with $\tilde{\mc{Y}}(\bar{\bs{\omega}}) \coloneqq \mc{X} \times \prod_{i \in \mc{I}} \tilde{\mc{C}}_i(\theta_i; \bar{\bs{\omega}})$, while $\mc{R}(\theta)$ by
\begin{equation}\label{eq:convex_feasible_set}
	\tilde{\mc{R}}(\theta; \bar{\bs{\omega}}) \coloneqq \{\bs{\omega} \in \tilde{\Omega}(\theta; \bar{\bs{\omega}}) \mid A_{\omega} \, \bs{\omega} - d = 0\}.
\end{equation}

Finally, by discarding constant terms and introducing $c_{\omega}(\bar{\bs{\omega}}) \coloneqq \col(\nabla J_0 (\bar{\bs{\varphi}}),0)$, the convexified version of $R(\theta)$ in \eqref{eq:KKT_into_leader_relaxed} reads as
\begin{equation}\label{eq:KKT_into_leader_relaxed_compact}
\tilde{R}\left(\theta; \bar{\bs{\omega}}\right) : \left\{
\begin{aligned}
&\underset{\blue{\bs{\omega} \in \tilde{\Omega}(\theta; \bar{\bs{\omega}})}}{\textrm{min}} & & c_{\omega}(\bar{\bs{\omega}})^\top \bs{\omega} + \frac{\sigma}{2} \| \bs{\omega} - \bar{\bs{\omega}} \|^2\\
&\hspace{.35cm}\textrm{ s.t. } & & A_{\omega} \, \bs{\omega} = d,
\end{aligned}	
\right.
\end{equation}
where we add a ``proximal-like'' term to the linearized cost function in \eqref{eq:KKT_into_leader_relaxed} with $\sigma > 0$. Hence, the cost function in \eqref{eq:KKT_into_leader_relaxed_compact}, namely $\tilde{J}_0(\bs{\omega}; \bar{\bs{\omega}}) \coloneqq c_{\omega}(\bar{\bs{\omega}})^\top \bs{\omega} + \frac{\sigma}{2} \| \bs{\omega} - \bar{\bs{\omega}} \|^2$, is characterized as follows.
\smallskip
\begin{lemma}\label{lemma:basic_J_0}
	The following statements hold true:
	\begin{itemize}
		\item[(i)] Given any $\bar{\bs{\omega}} \in \mc{R}(\theta)$, $\tilde{J}_0(\cdot \,; \bar{\bs{\omega}})$ is uniformly strongly convex on $\Phi \times \R^{2(m+p)}_{\geq 0}$, $p \coloneqq \sum_{i \in \mc{I}} p_i$, with coefficient $\sigma$;
		
		\item[(ii)] Given any $\bs{\omega} \in \mc{R}(\theta)$, $\nabla \tilde{J}_0(\bs{\omega}; \cdot )$ is uniformly Lipschitz continuous on $\mc{R}(\theta)$ with coefficient $\tilde{\kappa}_0 \coloneqq \kappa_0 + \sigma$.
	\end{itemize}
	\hfill$\square$
\end{lemma}
\begin{proof}
	$\textrm{(i)}$ The statement directly follows by applying the definition of uniform strong convexity on the set $\Phi \times \R^{2(m+p)}_{\geq 0}$.

	$\textrm{(ii)}$ Let $\bs{\omega}_1, \bs{\omega}_2 \in \mc{R}(\theta)$. For any given $\bs{\omega} \in \mc{R}(\theta)$, we have:
	$$
	\begin{aligned}
	\| \nabla \tilde{J}_0(\bs{\omega}; \bs{\omega}_1)& \!-\!  \nabla \tilde{J}_0(\bs{\omega}; \bs{\omega}_2) \| \!=\! \|c_{\omega}(\bs{\omega}_1) \!-\! c_{\omega}(\bs{\omega}_2) \!+\! \sigma (\bs{\omega}_2 \!-\! \bs{\omega}_1) \|\\
	&\hspace{-.6cm}\leq \| \col(\nabla J_0(\bs{\varphi}_1),0) \!-\! \col(\nabla J_0(\bs{\varphi}_2),0) \| \!+\! \sigma \|\bs{\omega}_1 - \bs{\omega}_2\| \\
	&\hspace{-.6cm}\leq (\kappa_0 + \sigma) \, \|\bs{\omega}_1 - \bs{\omega}_2\|.
	\end{aligned}
	$$
\end{proof}

\smallskip
\begin{remark}
	According to the structure of the vector $\bs{\omega}$, the coefficient $\sigma$ may be replaced with locally defined $\sigma_0, \, \sigma_\textrm{c}, \, \sigma_i > 0$, $i \in \mc{I}$, without affecting the results given in the remainder, see \cite[\S III.A]{scutari2014decomposition}. For simplicity, we adopt a unique, globally known parameter $\sigma$.
	\hfill$\square$
\end{remark}
\smallskip

Thus, given any $\bar{\bs{\omega}} \in \mc{R}(\theta)$, $\tilde{R}\left(\theta; \bar{\bs{\omega}}\right)$ in \eqref{eq:KKT_into_leader_relaxed_compact} admits a unique optimal solution associated with the mapping $\hat{\bs{\omega}} : \R^{s} \to \R^{s}$, with $s \coloneqq n_0 + n + 2(p + m)$, defined as follows:
\begin{equation}\label{eq:mapping}
	\hat{\bs{\omega}}(\bar{\bs{\omega}}) \coloneqq \underset{\bs{\omega} \in \tilde{\mc{R}}(\theta; \bar{\bs{\omega}})}{\textrm{argmin}} \; \tilde{J}_0(\bs{\omega}; \bar{\bs{\omega}}) .
\end{equation}

For computing an $\ell$-SE, we propose the iterative procedure summarized in Algorithm~\ref{alg:two_layers}, which is composed of two main loops and resorts on the so called SCA method. Specifically, once fixed the coefficients $\theta,\, \theta_i > 0$, for all $i \in \mc{I}$, at each iteration $k \in \N$, the outer loop is in charge of providing a feasible set of strategies $\bs{\omega}^k$, which are used to convexify $R(\theta)$ (\texttt{S1}). Then, after solving the inner loop by computing the optimal solution $\hat{\bs{\omega}}^{k} \coloneqq \hat{\bs{\omega}}(\bs{\omega}^k)$ to $\tilde{R}(\theta; \bs{\omega}^k)$ (\texttt{S2}), the outer loop updates the strategies $\bs{\omega}^{k+1}$ (\texttt{S3}) to find a new approximation $\tilde{R}(\theta; \bs{\omega}^{k+1})$, and the procedure repeats until a certain stopping criterion is met.
\begin{algorithm}[!t]
	\caption{Two-layer \gls{SCA} computation of \gls{l-SE}}\label{alg:two_layers}
	\DontPrintSemicolon
	\SetArgSty{}
	\SetKwFor{ForAll}{for all}{do}{end forall}
	\textbf{Initialization:} $\bs{\omega}^0 \in \mc{R}(\theta)$,  $\alpha > 0$\\
	\smallskip
	\textbf{Iteration $(k \in \N)$:} \\
	\begin{itemize}\setlength{\itemindent}{.75cm}
	\item[(\texttt{S1})] Convexify $R(\theta)$ to obtain $\tilde{R}(\theta; \bs{\omega}^{k})$ as in \eqref{eq:KKT_into_leader_relaxed_compact}\\
	\smallskip
	\item[(\texttt{S2})] Compute $\hat{\bs{\omega}}^{k}$, solution to $\tilde{R}(\theta; \bs{\omega}^{k})$\\
	\smallskip
	\item[(\texttt{S3})] Update $\bs{\omega}^{k+1} = (1 - \alpha) \bs{\omega}^{k} + \alpha \hat{\bs{\omega}}^{k}$
	\end{itemize}
\end{algorithm}
\subsection{Convergence analysis}
First, we characterize the sequence $(\bs{\omega}^k)_{k \in \N}$ generated by Algorithm~\ref{alg:two_layers} in terms of iterate feasibility. Then, we establish a key property of the mapping $\hat{\bs{\omega}}(\cdot)$, and finally we prove that $(\bs{\omega}^k)_{k \in \N}$ converges to an optimal solution to \eqref{eq:KKT_into_leader_relaxed}, generating an \gls{l-SE} of the hierarchical aggregative game \eqref{eq:single_prob}--\eqref{eq:single_prob_leader}, according to Definition~\ref{def:S-EQ}.
\smallskip
\begin{lemma}\label{lemma:iterate_feas}
	The following inclusions hold true:
	\begin{itemize}
		\item[(i)] $\tilde{\mc{R}}(\theta; \bar{\bs{\omega}}) \subseteq \mc{R}(\theta)$, for all $\bar{\bs{\omega}} \in \mc{R}(\theta)$;
		
		\item[(ii)] $\bs{\omega}^k \in \mc{R}(\theta)$.
	\end{itemize}
	\hfill$\square$	
\end{lemma}

\begin{proof}
	$\textrm{(i)}$ The upper approximation of the nonlinear constraints, which holds true for all $\bar{\bs{\omega}} \in \mc{R}(\theta)$, implies $\tilde{\mc{C}}(\theta; \bar{\bs{\omega}}) \subseteq \mc{C}(\theta)$ and $\tilde{\mc{C}}_i(\theta_i; \bar{\bs{\omega}}) \subseteq \mc{C}_i(\theta_i)$, $i \in \mc{I}$. Therefore, $\tilde{\Omega}(\theta; \bar{\bs{\omega}}) \subseteq \Omega(\theta)$, and in view of the definitions in \eqref{eq:feasible_set} and \eqref{eq:convex_feasible_set}, inclusion $\textrm{(i)}$ can be deduced.
	
	$\textrm{(ii)}$ First, in view of the approximation of the constraints, note that $\bs{\omega}^k \in \tilde{\mc{R}}(\theta; \bs{\omega}^k)$, for all $k \in \N$, with $\tilde{\mc{R}}(\theta; \bs{\omega}^k)$ convex subset of $\mc{R}(\theta)$. Then, the proof follows by induction by considering that $\bs{\omega}^{k+1}$ is a convex combination of $\hat{\bs{\omega}}^k \in \tilde{\mc{R}}(\theta; \bs{\omega}^k)$ and $\bs{\omega}^k$.
\end{proof}
\smallskip
\begin{lemma}\label{lemma:fundamental_mapping}
	For every $\bar{\bs{\omega}} \in \mc{R}(\theta)$, the vector $(\hat{\bs{\varphi}}(\bar{\bs{\omega}}) - \bar{\bs{\varphi}})$ is a descent direction for $J_0(\bs{\varphi})$ in $R(\theta)$, evaluated at $\bar{\bs{\varphi}}$\blue{, i.e., $(\bar{\bs{\varphi}} - \hat{\bs{\varphi}}(\bar{\bs{\omega}}))^\top \nabla J_0(\bar{\bs{\varphi}}) \geq \sigma \| \bar{\bs{\omega}} - \hat{\bs{\omega}}(\bar{\bs{\omega}}) \|^2 > 0$.}
	\hfill$\square$
		
		
\end{lemma}

\begin{proof}
	Given any $\bar{\bs{\omega}} \in \mc{R}(\theta)$, by definition, $\hat{\bs{\omega}}(\bar{\bs{\omega}})$ satisfies the minimum principle for \eqref{eq:KKT_into_leader_relaxed_compact}, i.e., $(\bs{\zeta} - \hat{\bs{\omega}}(\bar{\bs{\omega}}))^\top \nabla \tilde{J}_0(\hat{\bs{\omega}}(\bar{\bs{\omega}}); \bar{\bs{\omega}}) \geq 0$ for all $\bs{\zeta} \in \tilde{\mc{R}}(\theta; \bar{\bs{\omega}})$. From Lemma~\ref{lemma:iterate_feas}(ii), we choose $\bs{\zeta} = \bar{\bs{\omega}}$, and by adding and subtracting the term $(\bar{\bs{\omega}} - \hat{\bs{\omega}}(\bar{\bs{\omega}}))^\top \nabla \tilde{J}_0(\bar{\bs{\omega}}; \bar{\bs{\omega}})$, we obtain
	\begin{align*}
	(\bar{\bs{\omega}} - \hat{\bs{\omega}}(\bar{\bs{\omega}}))^\top &\nabla \tilde{J}_0(\bar{\bs{\omega}}; \bar{\bs{\omega}}) \geq \\ 
	&(\bar{\bs{\omega}} - \hat{\bs{\omega}}(\bar{\bs{\omega}}))^\top (\nabla \tilde{J}_0(\bar{\bs{\omega}}; \bar{\bs{\omega}}) -  \nabla \tilde{J}_0(\hat{\bs{\omega}}(\bar{\bs{\omega}}); \bar{\bs{\omega}}))
	\end{align*}
	By directly replacing $\nabla \tilde{J}_0(\bar{\bs{\omega}}; \bar{\bs{\omega}})$ with $c_\omega(\bar{\bs{\omega}}) = \col(\nabla J_0(\bar{\bs{\varphi}}), 0)$, the term on the left-hand side is equal to $(\bar{\bs{\varphi}} - \hat{\bs{\varphi}}(\bar{\bs{\omega}}))^\top \nabla J_0(\bar{\bs{\varphi}})$, while the one on the right-hand side, in view of Lemma~\ref{lemma:basic_J_0}(i), is bounded from below by $\sigma \| \bar{\bs{\omega}} - \hat{\bs{\omega}}(\bar{\bs{\omega}}) \|^2$, leading to
	$$
	(\bar{\bs{\varphi}} - \hat{\bs{\varphi}}(\bar{\bs{\omega}}))^\top \nabla J_0(\bar{\bs{\varphi}}) \geq \sigma \| \bar{\bs{\omega}} - \hat{\bs{\omega}}(\bar{\bs{\omega}}) \|^2.
	$$
\end{proof}
\smallskip

Before establishing the convergence to an $\ell$-SE for the sequence generated by Algorithm~\ref{alg:two_layers}, we recall a key result provided in \cite{scutari2017parallel}.
\smallskip
\begin{lemma}(\hspace{-.01cm}\cite[Th.~14]{scutari2017parallel})\label{lemma:scutari}
	Let $(\bs{\omega}^k)_{k \in \N}$ be the sequence generated by Algorithm~\ref{alg:two_layers} and assume that $\textrm{lim}_{k \to \infty} \, \| \hat{\bs{\omega}}(\bs{\omega}^k) - \bs{\omega}^k \| = 0$. Then, every limit point of $(\bs{\omega}^k)_{k \in \N}$ generated by Algorithm~\ref{alg:two_layers} is a stationary solution to $R(\theta)$.
	\hfill$\square$
\end{lemma}

\smallskip
\begin{theorem}\label{th:convergence}
	Let $\alpha$ in Algorithm~\ref{alg:two_layers} be chosen so that $\alpha \in (0,2 \sigma/\kappa_0)$. Then, the sequence $(\bs{\omega}^k)_{k \in \N}$ generated by Algorithm~\ref{alg:two_layers} converges to an optimal solution $\bs{\omega}^\ast$ to $R(\theta)$ in \eqref{eq:KKT_into_leader_relaxed}, which subvector $(y_0^\ast, \bs{x}^\ast)$ is an \gls{l-SE} of the hierarchical game in \eqref{eq:single_prob}--\eqref{eq:single_prob_leader}.
	\hfill$\square$
\end{theorem}

\begin{proof}
	By combining the descent lemma \cite[Prop.~A.24]{bertsekas1997nonlinear} and Lemma~\ref{lemma:fundamental_mapping}, the step (\texttt{S3}) in Algorithm~\ref{alg:two_layers} leads to:
	\begin{align*}
	J_0(\bs{\varphi}^{k+1}) &\leq J_0(\bs{\varphi}^{k}) + \alpha \nabla^\top J_0(\bs{\varphi}^{k}) (\hat{\bs{\varphi}}(\bs{\omega}^k) - \bs{\varphi}^k)\\ 
	&\hspace{4.2cm}+ \alpha^2\tfrac{\kappa_0}{2} \| \hat{\bs{\varphi}}(\bs{\omega}^k) - \bs{\varphi}^k \|^2\\
	\\[-2.5ex]
	&\leq J_0(\bs{\varphi}^{k}) - \alpha \left(\sigma - \alpha\tfrac{\kappa_0}{2} \right) \| \hat{\bs{\omega}}(\bs{\omega}^k) - \bs{\omega}^k \|^2,
	\end{align*}
	where the second inequality follows from $\| \hat{\bs{\omega}}(\bs{\omega}^k) - \bs{\omega}^k \| \ge \| \hat{\bs{\varphi}}(\bs{\omega}^k) - \bs{\varphi}^k \|$.
	If $\alpha < 2 \sigma/\kappa_0$, then $(J_0(\bs{\varphi}^k))_{k \in \N}$ shall converge to a finite value, since $J_0(\bs{\varphi}^k) \to -\infty$ can not happen in view of Standing Assumption~\ref{ass:continuity}. Thus, the convergence of $(J_0(\bs{\varphi}^k))_{k \in \N}$ implies $\textrm{lim}_{k \to \infty} \, \| \hat{\bs{\omega}}(\bs{\omega}^k) - \bs{\omega}^k \| = 0$, and therefore the bounded sequence $(\bs{\omega}^k)_{k \in \N} \in \mc{R}(\theta)$ in view of Lemma~\ref{lemma:iterate_feas}, and  has a limit point in $\mc{R}(\theta)$. From Lemma~\ref{lemma:scutari}, such a limit point is a stationary solution to $R(\theta)$, and since $(J_0(\bs{\varphi}^k))_{k \in \N}$ is a strictly decreasing sequence, no limit point can be a local maximum of $J_0$. Thus, $(\bs{\omega}^k)_{k \in \N}$ converges to an optimal solution $\bs{\omega}^\ast$ to \eqref{eq:KKT_into_leader_relaxed}, which subvector $(y_0^\ast, \bs{x}^\ast)$ is an \gls{l-SE} of the original hierarchical game in \eqref{eq:single_prob}--\eqref{eq:single_prob_leader}.
\end{proof}
\smallskip

\begin{remark}
	If the parameters $\sigma$ and $\kappa_0$ are not globally known, Theorem~\ref{th:convergence} can be equivalently restated  according to a vanishing step-size rule, i.e., $\alpha = \alpha^k$ that shall be chosen so that $\alpha^k \in (0,1]$, for all $k \in \N$, $\alpha^k \to 0$ and $\sum_{k \in \N} \alpha^k = + \infty$.
	\hfill$\square$
\end{remark}

\subsection{An augmented Lagrangian approach to solve the inner loop}

A scalable and privacy-preserving algorithm, suitable to solve (\texttt{S2}) in Algorithm~\ref{alg:two_layers} by exploiting the hierarchical structure of the original game, is the \gls{ADAL} method proposed in \cite{chatzipanagiotis2017convergence}. Since we are interested in finding the optimal solution to $\tilde{R}(\theta; \bs{\omega}^k)$, from now on we omit the dependence on $\bs{\omega}^k$ (unless differently specified) to alleviate the notation.

\begin{algorithm}[!t]
	\caption{\gls{ADAL} for (\texttt{S2}) of Algorithm~\ref{alg:two_layers}}\label{alg:ADAL}
	\DontPrintSemicolon
	\SetArgSty{}
	\SetKwFor{ForAll}{for all}{do}{end forall}
	\textbf{Initialization:} $\eta(0) \in \R^s$, $\tau,\, \rho > 0$\\
	\smallskip
	\textbf{Iteration $(t \in \N)$:} \\
	\begin{itemize}
		\item Leader:\\
		\medskip
		$
		\left\{
		\begin{aligned}
		y^\star_0(t) &= \underset{y_0 \in \mc{Y}_0}{\textrm{argmin}} \; \hat{\mc{L}}^k_{\ell}(y_0, \eta(t), z_{\textrm{f}}(t), z_{\textrm{c}}(t))\\
		z_\ell(t+1) &= z_\ell(t) + \tau ( A_\ell \, y^\star_0(t) - z_\ell(t) )
		\end{aligned}
		\right.
		$
		
		\medskip
		
		\item Followers:\\
		\medskip
		$
		\left\{
		\begin{aligned}
		\bs{y}^\star(t) &= \underset{\bs{y} \in \tilde{\mc{Y}}^k}{\textrm{argmin}} \; \hat{\mc{L}}^k_{\textrm{f}}(\bs{y}, \eta(t), z_\ell(t), z_{\textrm{c}}(t))\\
		z_{\textrm{f}}(t+1) &= z_{\textrm{f}}(t) + \tau ( A_{\textrm{f}} \, \bs{y}^\star(t) - z_{\textrm{f}}(t) )
		\end{aligned}
		\right.
		$
		
		\medskip
		
		\item Coordinator:\\
		\medskip
		$
		\left\{
		\begin{aligned}
		\nu^\star(t) &= \underset{\nu \in \tilde{\mc{C}}^k(\theta)}{\textrm{argmin}} \;  \hat{\mc{L}}^k_{\textrm{c}}(\nu, \eta(t), z_{\textrm{f}}(t), z_{\ell}(t))\\
		z_{\textrm{c}}(t+1) &= z_{\textrm{c}}(t) + \tau \left( A_{\textrm{c}} \, \nu^\star(t) - z_{\textrm{c}}(t) \right)
		\end{aligned}
		\right.
		$
	\end{itemize}
	
	\vspace{-.2cm}
	$$
	\eta(t+1) = \eta(t) + \rho \tau \left( z_{\textrm{f}}(t+1) + z_\ell(t+1) + z_{\textrm{c}}(t+1) - d  \right)
	$$
\end{algorithm}

Thus, at every iteration $k \in \N$ of the outer loop, the Lagrangian function associated to \eqref{eq:KKT_into_leader_relaxed_compact} is defined as
\begin{equation}\label{eq:Lagrangian}
\mc{L}^k(\bs{\omega},\nu) = (c_{\omega}^k)^\top \bs{\omega} + \frac{\sigma}{2} \|\bs{\omega} - \bs{\omega}^k \|^2 + \eta^\top (A_\omega \, \bs{\omega} - d),
\end{equation}
where $c_{\omega}^k \coloneqq c_{\omega}(\bs{\omega}^k)$, and $\nu \in \R^{l}$, $l \coloneqq n + m + p$, is the dual variable associated with the linear equality constraints. \blue{Note that} the Lagrangian in \eqref{eq:Lagrangian} can be rewritten as the sum of terms associated to different entities, which happens to correspond to leader, the set of followers, and a central coordinator, respectively. In details, we define $\mc{L}^k_\ell \coloneqq (c^k_\ell)^\top y_0 + \tfrac{\sigma}{2} \| y_0 - y^k_0 \|^2 + \eta^\top A_\ell \, y_0$, $\mc{L}^k_{\textrm{f}} \coloneqq (c^k_{\textrm{f}})^\top \bs{y} + \tfrac{\sigma}{2} \| \bs{y} - \bs{y}^k \|^2 + \eta^\top A_{\textrm{f}} \, \bs{y}$, and $\mc{L}^k_{\textrm{c}} \coloneqq \tfrac{\sigma}{2} \| \nu - \nu^k \|^2 + \eta^\top A_{\textrm{c}} \, \nu$.
In light of \cite{chatzipanagiotis2017convergence}, we augment each one of these terms as, e.g., $ \hat{\mc{L}}^k_{\textrm{f}} \coloneqq \mc{L}^k_{\textrm{f}} + \frac{\rho}{2} \| A_{\textrm{f}} \, \bs{y} + A_\ell \, y_0 + A_{\textrm{c}} \, \nu  - d \|^2$ ($\hat{\mc{L}}^k_{\ell}$  and $\hat{\mc{L}}^k_{\textrm{c}}$ are identical), where $\rho >0$ is a penalty term to be designed freely.

The main steps of the proposed semi-decentralized procedure are summarized in Algorithm~\ref{alg:ADAL}, where we emphasize that each augmented Lagrangian term depends on the linearization at the current outer iteration $k \in \N$. Specifically, at every iteration $t \in \N$ of the inner loop, the ADAL requires that the followers, the leader and the central coordinator compute in parallel a minimization step of the local augmented Lagrangian.
Here, $z_\ell \coloneqq A_\ell \, y_0$, $z_{\textrm{f}} \coloneqq A_{\textrm{f}} \, \bs{y}$ and $z_{\textrm{c}} \coloneqq A_{\textrm{c}} \, \nu$ are auxiliary variables introduced for privacy purposes and, given some $\tau > 0$, are locally updated.
Finally, the central coordinator, which in some practical applications may eventually coincide with the leader, gathers $z_\ell(t+1)$ and $z_{\textrm{f}}(t+1)$ from the leader and followers, and updates the dual variable.

\smallskip
\begin{proposition}
	Let $\rho > 0$ be sufficiently large and $\tau \in (0, r^{-1}_{\textrm{max}})$, where $r_{\textrm{max}}$ corresponds to the maximum degree among the constraints in \eqref{eq:convex_feasible_set}. Then, the sequence $(\bs{\omega}(t))_{t \in \N}$ generated by Algorithm~\ref{alg:ADAL} converges to the minimizer of $\tilde{R}(\theta; \bs{\omega}^k)$, for all $k \in \N$.
	\hfill$\square$
\end{proposition}

\begin{proof}
	The proof follows by noticing that $\tilde{R}(\theta; \bs{\omega}^k)$ satisfies the assumptions in \cite[Th.~2]{chatzipanagiotis2017convergence}, for all $k \in \N$. Specifically, $\tilde{\mc{R}}(\theta; \bs{\omega}^k)$ is a closed and convex set, $\tilde{J}_0(\bs{\omega}; \bs{\omega}^k)$ is inf-compact and each one of its terms is twice continuously differentiable. Finally, Lemma~\ref{lemma:scholtes} provides the local LICQ for $R(\theta)$, directly inherited by $\tilde{R}(\theta; \bs{\omega}^k)$.
\end{proof}

\smallskip
\begin{remark}
	For simplicity, we adopt a common $\tau$ to update the auxiliary variables $z_{\ell}$, $z_{\textrm{f}}$ and $z_{\textrm{c}}$. In principle, each entity involved within the ADAL in Algorithm~\ref{alg:ADAL} can locally set its own step size according to the degree of each constraint in \eqref{eq:KKT_into_leader_relaxed_compact}, see \cite[\S II.A]{chatzipanagiotis2017convergence}.
	\hfill$\square$
\end{remark}

\section{Numerical case study: Charging coordination of plug-in electric vehicles}
\subsection{Numerical simulation setup}
We consider a set of \glspl{PEV} (followers), $\mathcal{I} \coloneqq \left\{ {1,2, \ldots ,N} \right\}$, \blue{which has to be charged} over a certain horizon $\mc{T} \coloneqq \{1,\ldots,T\}$. All \glspl{PEV} are connected to an aggregator (leader, e.g., a retailer), which manages the energy requirements of the fleet by purchasing the electricity from the wholesale energy market. Let us define $x_i \coloneqq \col(\{x_i^j\}_{j \in \mc{T}})$, and $p \coloneqq \col(\{p^j\}_{j \in \mc{T}})$ as the amount of requested energy by the fleet and the price of energy over time, \blue{i.e.}, the strategy of the $i$-th follower and of the leader, respectively. For every \gls{PEV} $i \in \mc{I}$, we consider the cost function 
$
J_i(p, \bs{x}) = q_i x_i^\top x_i+c_i^\top x_i -\left( -s_i x_i^\top x_i +\kappa_i^\top x_i + p^\top x_i \right) +\delta \left\| x_i-\sigma(\bs{x})\right\|^2,
$
where  $\bs{x} \coloneqq \col(\{x_i\}_{i \in \mc{I}})$, $q_i$, $c_i > 0$ depend on the nominal voltage and on the capacity loss of \blue{each} battery, while $\kappa_i$, $s_i > 0$ model the battery size and the satisfaction of the $i$-th \gls{PEV} for charging the amount $x_i$.
\blue{Moreover, the term} $(q_i x_i^\top x_i+c_i^\top x_i)$ \blue{denotes} the battery degradation cost, $( -s_i x_i^\top x_i +\kappa_i^\top x_i + p^\top x_i )$ the benefit for charging \cite{tajeddini2018mean}, and $\blue{\delta} \|x_i-\sigma(\bs{x})\| ^2$ a penalty for deviating from the average charging profile, $\sigma(\bs{x}) \coloneqq \tfrac{1}{N} \bsone^\top \bs{x}$, \blue{with} $\delta > 0$. \blue{On the other hand,} the leader \blue{aims at maximizing the following cost function}
\begin{equation}\label{CostLeaderSim}
J_0(p, \bs{x}) = -p^\top \left( D + \sigma(\bs{x})\right),
\end{equation}
which represents the economic benefit for charging the \glspl{PEV}, where $D \in \R^T$ is the total non-\gls{PEV} demand over time.
We assume that the net energy available for the \glspl{PEV} is fixed\blue{, and therefore} the overall \gls{PEV} demand \blue{shall meet the} capacity constraint
$
 \tfrac{1}{N} \bsone^\top \bs{x} \le C,
$
for some $C > 0$. Furthermore, we assume that\blue{, at every time step,} $\blue{x_i \in [\underline{x}_i,  \overline{x}_i]}$, for all $i \in \mc{I}$.
Thus, given the amount of energy requested by the PEVs, the retailer chooses a price $p$ per unit of energy, \blue{with $p \in [0, \overline{p}]$}, aiming at maximizing its revenue in \eqref{CostLeaderSim}.

For the numerical simulations, we consider $N = 10^4$ PEVs, a charging horizon discretized into $T = 24$ time intervals, $q_i = 1.2\cdot10^{-3}$, $c_i = 0.11$, while $\kappa_i$ and $s_i$ are randomly drawn from $\mc{N}(12, 2)$ and $\mc{U}(0.02, 0.1)$ respectively, while the capacity upper bound $C$ is equal to $1.5$ from $11$pm to $8$am, and to $0.5$ for the rest of the day. The convergence behavior of Algorithm~\ref{alg:two_layers} over $10$ experiments is shown in Fig.~\ref{error}. During the numerical simulations, the inner loop takes between $50$ and $75$ iterations (on average) to meet a predefined stopping condition, and above $10^2$ experiments, we did not experienced any influence on the outer loop convergence behavior\blue{. For this latter, in view of the fact that $\textrm{lim}_{k \to \infty} \, \| \hat{\bs{\omega}}(\bs{\omega}^k) - \bs{\omega}^k \| = 0$, we have chosen $\| \hat{\bs{\omega}}^k - \hat{\bs{\omega}}^{k-1} \| \leq 10^{-4}$ as stopping criterion}.

\begin{algorithm}[!t]
	\caption{Two-layer na\"ive method for \gls{l-SE} computation}\label{alg:naive_alg}
	\DontPrintSemicolon
	\SetArgSty{}
	\SetKwFor{ForAll}{for all}{do}{end forall}
	\textbf{Initialization:} $y_0(0) \in \mc{Y}_0$\\
	\smallskip
	\textbf{Iteration $(k \in \N)$:} \\
	\begin{itemize}\setlength{\itemindent}{.75cm}
	\item[(\texttt{S1})] Compute an \gls{v-GNE}, $\bs{x}(k)$, for the game in \eqref{eq:single_prob}\\
	\smallskip
	\item[(\texttt{S2})] Compute $y^\ast_0(k)$, solution to \eqref{eq:single_prob_leader}\\
	\smallskip
	\item[(\texttt{S3})] Update $y_0(k) \coloneqq (1-\beta(k)) y_0(k-1) + \beta(k) y^\ast_0(k)$
	\end{itemize}
\end{algorithm}

The procedure proposed in Algorithm~\ref{alg:two_layers} is then compared with the simplest na\"ive method for possibly computing an $\ell$-SE, which main steps are summarized in Algorithm~\ref{alg:naive_alg}. Specifically, given the strategy of the leader at the previous step, the followers compute an \gls{v-GNE} of the game in \eqref{eq:single_prob}, and send their strategy back to the leader (\texttt{S1}). In turn, the leader first solves its optimization problem in \eqref{eq:single_prob_leader} with solution $y^\ast_0(k)$ (\texttt{S2}), and then updates its strategy taking a convex combination between $y^\ast_0(k)$ and the strategy at the previous step, where the parameter $\beta(k) \in [0,1]$ introduces a possible inertia (\texttt{S3}).
Note that, albeit rather intuitive, this na\"ive algorithm has no converge guarantees. However, in our numerical experience, by considering the cost function in \eqref{CostLeaderSim} for the leader and setting $\beta(k) = 1/k$, Algorithm~\ref{alg:naive_alg} apparently shows a slower convergent behavior compared with the proposed Algorithm~\ref{alg:two_layers}, as depicted in Fig.~\ref{error} over $10$ numerical experiments.


\begin{figure}
	\centering
	\includegraphics[width=.99\columnwidth]{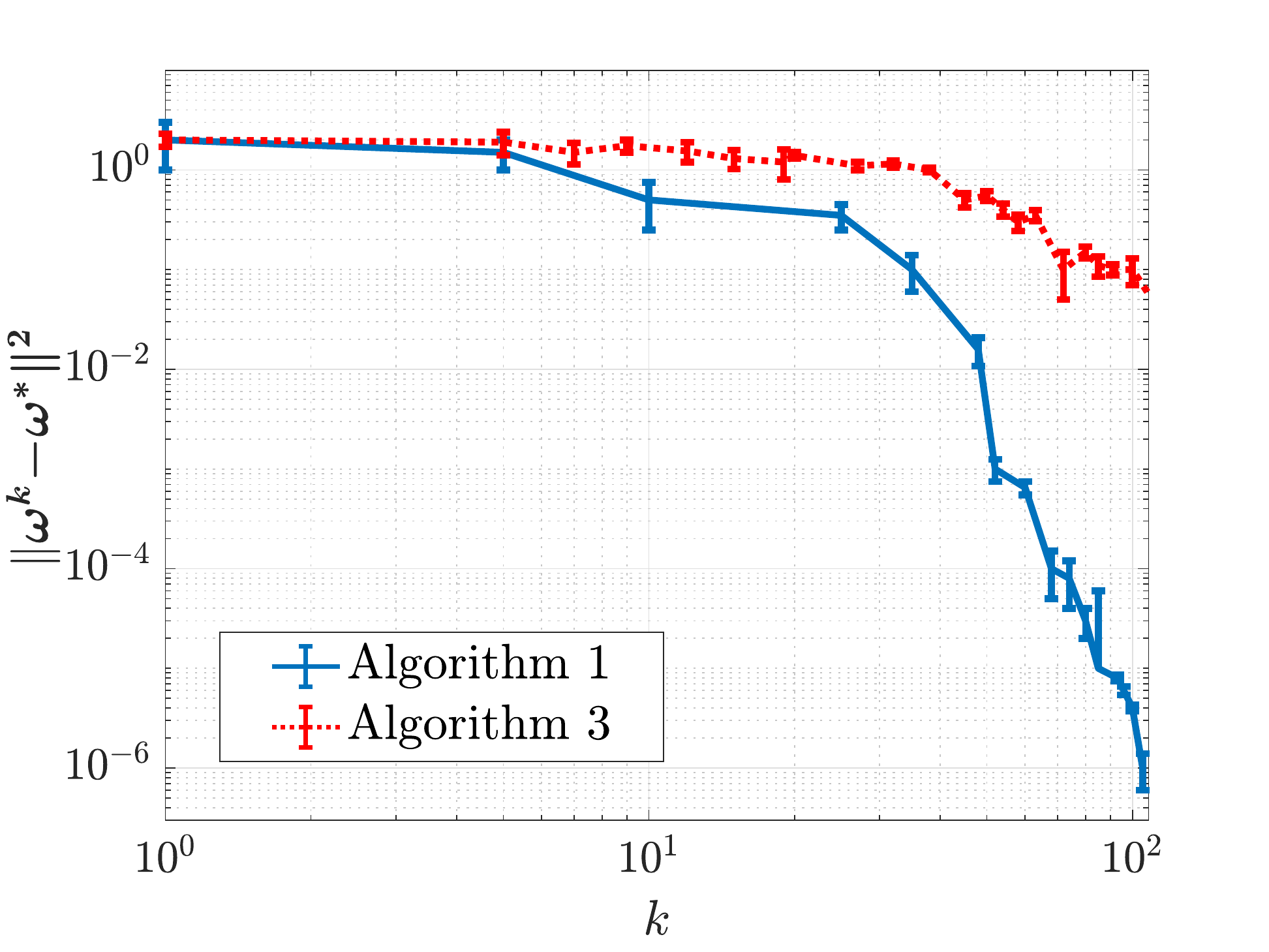}
	\caption{Comparison of the convergence behavior between Algorithm~\ref{alg:two_layers} (solid blue line) and \ref{alg:naive_alg} (dotted red line).}
	\label{error}
\end{figure}


\subsection{The trade-off between the leader and the followers}\label{subsec:tradeoff}
Finally, we highlight the trade-off role played by the relaxation parameter $\theta$ in \eqref{eq:KKT_into_leader_relaxed}. In fact, for $\theta$ sufficiently large, the leader has a larger feasible set while, on the other hand, the followers are farther away from an \gls{v-GNE}, since the complementarity condition is not exactly satisfied.
Therefore, the larger the $\theta$, the lower the optimal cost of the leader, and possibly the larger the optimal cost of each follower. Vice versa, the smaller $\theta$, the higher the optimal cost of the leader, because his feasible set shrinks, and possibly the lower the optimal cost of each follower, since the equilibrium condition is closer to being satisfied. This behavior is essentially confirmed in Fig.~\ref{theta} where, for ease of visualization, we show the normalized benefit of the leader ($J_0^\star(\theta)$) and the normalized maximum disadvantage among the followers ($\Delta J^\star(\theta)$) as $\theta$ increases. Specifically, for each $\theta \in [\underline{\theta},1]$, we compute an $\ell$-SE, and we denote with $J_0^\star(\theta)$ the corresponding optimal cost for the leader. For the followers, we introduce and show the maximum relative disadvantage with respect to a near-equilibrium condition, i.e.,
$
\Delta J^\star(\theta) \coloneqq \textrm{max}_{i \in \mathcal{I}} \;  J_i^\star(\theta) - J_i^{\star}(\underline{\theta}),
$ 
where, for a given $\theta$, $J_i^\star(\theta)$ is the optimal cost for the $i$-th follower, while in this case we set $\underline{\theta}$ equal to $10^{-6}$.

\begin{figure}
	\centering
	\includegraphics[width=.99\columnwidth]{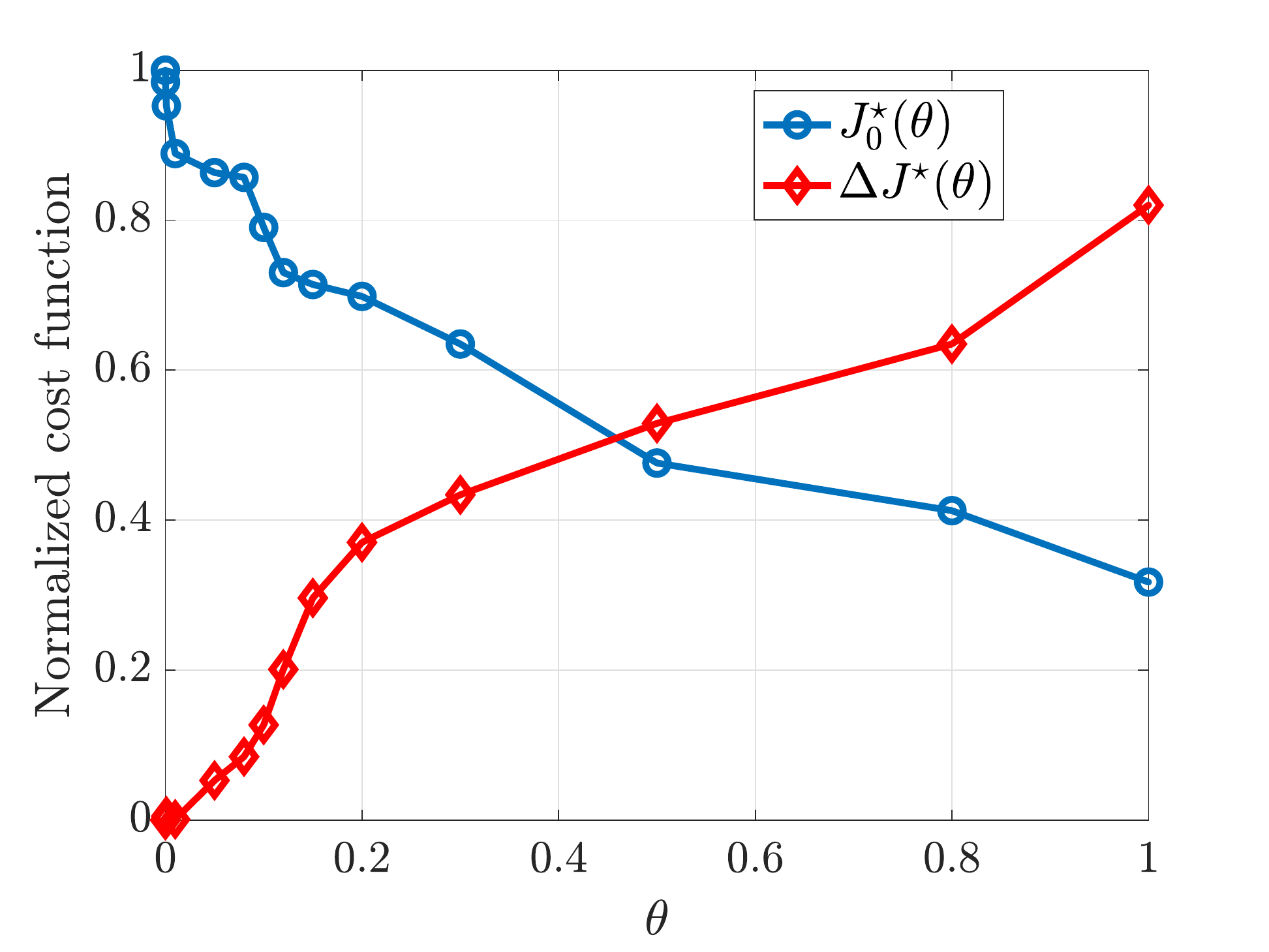}
	\caption{Trade-off role played by the regularization parameter $\theta$.}
	\label{theta}
\end{figure}

\section{Conclusion}

We have considered a multi-agent, hierarchical equilibrium problem with one leader and multiple followers, with possibly nonconvex data for the leader, convex-quadratic objective functions and linear constraints for the followers, and overall an aggregative structure.
In this setup, a local Stackelberg equilibrium can be approximated arbitrarily close via the relaxation of the complementarity condition that represents the equilibrium among the followers. In turn, the relaxed problem can be solved via a two-layer algorithm, which - thanks to the aggregative structure - requires semi-decentralized computations and information exchange.

\bibliographystyle{IEEEtran}
\bibliography{stackeq.bib}

\begin{thebibliography}{10}
\providecommand{\url}[1]{#1}
\csname url@samestyle\endcsname
\providecommand{\newblock}{\relax}
\providecommand{\bibinfo}[2]{#2}
\providecommand{\BIBentrySTDinterwordspacing}{\spaceskip=0pt\relax}
\providecommand{\BIBentryALTinterwordstretchfactor}{4}
\providecommand{\BIBentryALTinterwordspacing}{\spaceskip=\fontdimen2\font plus
\BIBentryALTinterwordstretchfactor\fontdimen3\font minus
  \fontdimen4\font\relax}
\providecommand{\BIBforeignlanguage}[2]{{%
\expandafter\ifx\csname l@#1\endcsname\relax
\typeout{** WARNING: IEEEtran.bst: No hyphenation pattern has been}%
\typeout{** loaded for the language `#1'. Using the pattern for}%
\typeout{** the default language instead.}%
\else
\language=\csname l@#1\endcsname
\fi
#2}}
\providecommand{\BIBdecl}{\relax}
\BIBdecl

\bibitem{liu2019sinr}
H.~Liu, ``{SINR}-based multi-channel power schedule under {DoS} attacks: {A}
  {S}tackelberg game approach with incomplete information,'' \emph{Automatica},
  vol. 100, pp. 274--280, 2019.

\bibitem{TongwenChen2018}
L.~Yuzhe, S.~Dawei, and C.~Tongwen, ``False data injection attacks on networked
  control systems: A {Stackelberg game} analysis,'' \emph{IEEE Transactions on
  Automatic Control}, vol.~63, no.~10, pp. 3503--3509, 2018.

\bibitem{motalleb2019networked}
M.~Motalleb, P.~Siano, and R.~Ghorbani, ``Networked {S}tackelberg competition
  in a demand response market,'' \emph{Applied Energy}, vol. 239, pp. 680--691,
  2019.

\bibitem{chen2017stackelberg}
J.~Chen and Q.~Zhu, ``A {Stackelberg} game approach for two-level distributed
  energy management in smart grids,'' \emph{IEEE Transactions on Smart Grid},
  2017.

\bibitem{mendoza2019online}
F.~G. Mendoza, D.~Bauso, and G.~Konstantopoulos, ``Online pricing via
  {S}tackelberg and incentive games in a micro-grid,'' in \emph{2019 18th
  European Control Conference (ECC)}.\hskip 1em plus 0.5em minus 0.4em\relax
  IEEE, 2019, pp. 3520--3525.

\bibitem{hirose2019comparing}
K.~Hirose and T.~Matsumura, ``Comparing welfare and profit in quantity and
  price competition within {S}tackelberg mixed duopolies,'' \emph{Journal of
  Economics}, vol. 126, no.~1, pp. 75--93, 2019.

\bibitem{groot2017hierarchical}
N.~Groot, G.~Zaccour, and B.~De~Schutter, ``Hierarchical game theory for
  system-optimal control: applications of reverse {Stackelberg} games in
  regulating marketing channels and traffic routing,'' \emph{IEEE Control
  Systems Magazine}, vol.~37, no.~2, pp. 129--152, 2017.

\bibitem{luo1996mathematical}
Z.-Q. Luo, J.-S. Pang, and D.~Ralph, \emph{Mathematical programs with
  equilibrium constraints}.\hskip 1em plus 0.5em minus 0.4em\relax Cambridge
  University Press, 1996.

\bibitem{scheel2000mathematical}
H.~Scheel and S.~Scholtes, ``Mathematical programs with complementarity
  constraints: {S}tationarity, optimality, and sensitivity,'' \emph{Mathematics
  of Operations Research}, vol.~25, no.~1, pp. 1--22, 2000.

\bibitem{jongen1991nonlinear}
H.~T. Jongen and G.-W. Weber, ``Nonlinear optimization: characterization of
  structural stability,'' \emph{Journal of Global Optimization}, vol.~1, no.~1,
  pp. 47--64, 1991.

\bibitem{simaan1973stackelberg}
M.~Simaan and J.~Cruz, ``A {S}tackelberg solution for games with many
  players,'' \emph{IEEE Transactions on Automatic Control}, vol.~18, no.~3, pp.
  322--324, 1973.

\bibitem{kydland1977equilibrium}
F.~Kydland, ``Equilibrium solutions in dynamic dominant-player models,''
  \emph{Journal of Economic Theory}, vol.~15, no.~2, pp. 307--324, 1977.

\bibitem{kydland1975noncooperative}
------, ``Noncooperative and dominant player solutions in discrete dynamic
  games,'' \emph{International Economic Review}, pp. 321--335, 1975.

\bibitem{bacsar2002stackelberg}
T.~Ba{\c{s}}ar and R.~Srikant, ``A {S}tackelberg network game with a large
  number of followers,'' \emph{Journal of Optimization Theory and
  Applications}, vol. 115, no.~3, pp. 479--490, 2002.

\bibitem{kebriaei2017discrete}
H.~Kebriaei and L.~Iannelli, ``Discrete-time robust hierarchical
  linear-quadratic dynamic games,'' \emph{IEEE Transactions on Automatic
  Control}, vol.~63, no.~3, pp. 902--909, 2017.

\bibitem{tushar2012economics}
W.~Tushar, W.~Saad, H.~V. Poor, and D.~B. Smith, ``Economics of electric
  vehicle charging: {A} game theoretic approach,'' \emph{IEEE Transactions on
  Smart Grid}, vol.~3, no.~4, pp. 1767--1778, 2012.

\bibitem{zou2017decentralized}
S.~Zou, I.~Hiskens, and Z.~Ma, ``Decentralized coordination of controlled loads
  and transformers in a hierarchical structure,'' \emph{IFAC World Congress},
  vol.~50, no.~1, pp. 5560--5566, 2017.

\bibitem{kulkarni2015existence}
A.~A. Kulkarni and U.~V. Shanbhag, ``An existence result for hierarchical
  {Stackelberg v/s Stackelberg} games,'' \emph{IEEE Transactions on Automatic
  Control}, vol.~60, no.~12, pp. 3379--3384, 2015.

\bibitem{leyffer2010solving}
S.~Leyffer and T.~Munson, ``Solving multi-leader--common-follower games,''
  \emph{Optimisation Methods \& Software}, vol.~25, no.~4, pp. 601--623, 2010.

\bibitem{hobbs2001linear}
B.~Hobbs, ``Linear complementarity models of {Nash-Cournot} competition in
  bilateral and {POOLCO} power markets,'' \emph{IEEE Transactions on Power
  Systems}, vol.~16, no.~2, pp. 194--202, 2001.

\bibitem{ehrenmann2009comparison}
A.~Ehrenmann and K.~Neuhoff, ``A comparison of electricity market designs in
  networks,'' \emph{Operations Research}, vol.~57, no.~2, pp. 274--286, 2009.

\bibitem{su2004sequential}
C.-L. Su, ``A sequential {NCP} algorithm for solving equilibrium problems with
  equilibrium constraints,'' \emph{Manuscript, Department of Management Science
  and Engineering, Stanford University, Stanford, CA}, 2004.

\bibitem{facchinei2007finite}
F.~Facchinei and J.-S. Pang, \emph{Finite-dimensional variational inequalities
  and complementarity problems}.\hskip 1em plus 0.5em minus 0.4em\relax
  Springer Science \& Business Media, 2007.

\bibitem{scholtes2001convergence}
S.~Scholtes, ``Convergence properties of a regularization scheme for
  mathematical programs with complementarity constraints,'' \emph{SIAM Journal
  on Optimization}, vol.~11, no.~4, pp. 918--936, 2001.

\bibitem{fletcher2006local}
R.~Fletcher, S.~Leyffer, D.~Ralph, and S.~Scholtes, ``Local convergence of
  {SQP} methods for mathematical programs with equilibrium constraints,''
  \emph{SIAM Journal on Optimization}, vol.~17, no.~1, pp. 259--286, 2006.

\bibitem{scutari2014decomposition}
G.~Scutari, F.~Facchinei, P.~Song, D.~P. Palomar, and J.-S. Pang,
  ``Decomposition by partial linearization: Parallel optimization of
  multi-agent systems,'' \emph{IEEE Transactions on Signal Processing},
  vol.~62, no.~3, pp. 641--656, 2014.

\bibitem{scutari2017parallel}
G.~Scutari, F.~Facchinei, and L.~Lampariello, ``Parallel and distributed
  methods for constrained nonconvex optimization - {P}art {I}: {T}heory,''
  \emph{IEEE Transactions on Signal Processing}, vol.~65, no.~8, pp.
  1929--1944, 2017.

\bibitem{hu2007using}
X.~Hu and D.~Ralph, ``Using {EPEC}s to model bilevel games in restructured
  electricity markets with locational prices,'' \emph{Operations Research},
  vol.~55, no.~5, pp. 809--827, 2007.

\bibitem{facchinei2007gen}
F.~Facchinei, A.~Fischer, and V.~Piccialli, ``On generalized {N}ash games and
  variational inequalities,'' \emph{Operations Research Letters}, vol.~35,
  no.~2, pp. 159--164, 2007.

\bibitem{bertsekas1997nonlinear}
D.~P. Bertsekas, ``Nonlinear programming,'' \emph{Journal of the Operational
  Research Society}, vol.~48, no.~3, pp. 334--334, 1997.

\bibitem{chatzipanagiotis2017convergence}
N.~Chatzipanagiotis and M.~M. Zavlanos, ``On the convergence of a distributed
  augmented {L}agrangian method for nonconvex optimization,'' \emph{IEEE
  Transactions on Automatic Control}, vol.~62, no.~9, pp. 4405--4420, 2017.

\bibitem{tajeddini2018mean}
M.~A. Tajeddini and H.~Kebriaei, ``A mean-field game method for decentralized
  charging coordination of a large population of plug-in electric vehicles,''
  \emph{IEEE Systems Journal}, vol.~13, no.~1, pp. 854--863, 2018.

\end{thebibliography}

\end{document}